\documentclass[a4,12pt]{elsart}
\usepackage{ifpdf}
\usepackage{amsmath,amsfonts,latexsym,amssymb}
\usepackage{graphicx,indentfirst}
\ifpdf

\usepackage[%
  pdftitle={Instructions for use of the document class
    elsart},%
  pdfauthor={Simon Pepping},%
  pdfsubject={The preprint document class elsart},%
  pdfkeywords={instructions for use, elsart, document class},%
  pdfstartview=FitH,%
  bookmarks=true,%
  bookmarksopen=true,%
  breaklinks=true,%
  colorlinks=true,%
  linkcolor=blue,anchorcolor=blue,%
  citecolor=blue,filecolor=blue,%
  menucolor=blue,pagecolor=blue,%
  urlcolor=blue]{hyperref}
\else
\usepackage[%
  breaklinks=true,%
  colorlinks=true,%
  linkcolor=blue,anchorcolor=blue,%
  citecolor=blue,filecolor=blue,%
  menucolor=blue,pagecolor=blue,%
  urlcolor=blue]{hyperref}
\fi

\makeatletter
\def\elsartstyle{%
    \def\normalsize{\@setfontsize\normalsize\@xiipt{14.5}}
    \def\small{\@setfontsize\small\@xipt{13.6}}
    \let\footnotesize=\small
    \def\large{\@setfontsize\large\@xivpt{18}}
    \def\Large{\@setfontsize\Large\@xviipt{22}}
    \skip\@mpfootins = 18\p@ \@plus 2\p@
    \normalsize
               }
\@ifundefined{square}{}{} \makeatother

\newtheorem{theorem}{Theorem}
\newtheorem{lemma}[theorem]{Lemma}
\newtheorem{corollary}[theorem]{Corollary}
\newtheorem{proposition}[theorem]{Proposition}
\newcommand{\fd}{\mathbb{F}}
\newcommand{\z}{\mathbb{Z}}

\begin{document}
%%%%%%%%%%%%%%%%%%%%%%%%%%%%%%%%%%%%%%%%%%%%%%%%%%%%%%%%%%%%%%%%%%%%%%%%%%%%%%%%%%%%%%%%%%%%%%%%%%%%%%%%%%%%%%%%%%%555
\begin{minipage}{\textwidth}
\elsartstyle
\parskip 12pt
\renewcommand{\thempfootnote}{\fnsymbol{mpfootnote}}
\leftskip=2pc
\begin{center}
{\LARGE Infinite Families of Recursive Formulas Generating Power
Moments of Kloosterman Sums: Symplectic Case \par}
\large
Dae San Kim\\[12pt]
\small\itshape Department of Mathematics, Sogang University, Seoul
121-742, Korea
\end{center}

\bigskip
\leftskip=0pt

\hrule\vskip 8pt
\begin{small}
{\bfseries Abstract}
\parindent 1em

In this paper, we construct two infinite families of binary linear
codes  associated with double cosets with respect to certain maximal
parabolic subgroup  of the symplectic group $Sp(2n,q)$.  Here $q$ is
a power of two. Then we obtain an infinite family of recursive
formulas for the power moments of Kloosterman sums and those of
$2$-dimensional Kloosterman sums in terms of the frequencies of
weights in the codes. This is done via Pless power moment identity
and by utilizing the explicit expressions of exponential sums over
those double cosets  related to the evaluations of ``Gauss sums" for
the symplectic  groups $Sp(2n,q)$.

\end{small}
\vskip 10pt\hrule

\leftskip=0pt

\vspace{24pt}
\renewcommand{\thempfootnote}{\astsymbol{mpfootnote}}
\footnotetext[1]{Corresponding author.} \setbox0=\hbox{\footnotesize
1} \edef\thempfootnote{\hskip\wd0} \footnotetext[0]{\textit{Email
adress:} dskim@sogang.ac.kr
  (Dae San Kim).}

%%%%%%%%%%%%%%%%%%%%%%%%%%%%%%%%%%%%%%%%%%%%%%%%%%%%%%%%%%%%%%%%%%%%
\section{Introduction}
%%%%%%%%%%%%%%%%%%%%%%%%%%%%%%%%%%%%%%%%%%%%%%%%%%%%%%%%%%%%%%%%%%%%
Let $\psi$ be a nontrivial additive character of the finite field
$\mathbb{F}_{q}$ with $q=p^{r}$ elements ($p$ a prime), and let $m$
be a positive integer. Then the $m$-dimensional  Kloosterman sum
$K_{m}(\psi;a)$(\cite{RH}) is defined by

\begin{align*}
K_{m}(\psi;a)=\sum_{\alpha_{1},\cdots,\alpha_{m} \in
\mathbb{F}_{q}^{*}}\psi(\alpha_{1}+\cdots+\alpha_{m}+a\alpha_{1}^{-1}\cdots\alpha_{m}^{-1})\\
(a \in \mathbb{F}_{q}^{*}).\\
\end{align*}

\end{minipage}
\bigskip

In particular, if $m=1$, then $K_{1}(\psi;a)$ is simply denoted by
$K(\psi;a)$, and is called the Kloosterman sum. The Kloosterman sum
was introduced in 1926 to give an estimate for the Fourier
coefficients of modular forms (cf. \cite{HDK}, \cite{JH}). It has
also been studied to solve  various problems in coding theory and
cryptography over finite fields of characteristic two (cf.
\cite{PTV}, \cite{HPT}).

For each nonnegative integer $h$, by $MK_{m}(\psi)^{h}$ we will
denote the $h$-th moment of the $m$-dimensional Kloosterman sum
$K_{m}(\psi;a)$. Namely, it is given by
\begin{equation*}
 MK_{m}(\psi)^{h}=\sum_{a \in \mathbb{F}_{q}^{*}}K_{m}(\psi;a)^{h}.
 \end{equation*}

If $\psi=\lambda$ is the canonical additive character of
$\mathbb{F}_{q}$, then $MK_{m}(\lambda)^{h}$ will be simply denoted
by $MK_{m}^{h}$. If further $m=1$, for brevity $MK_{1}^{h}$ will be
indicated by $MK^{h}$.

Explicit computations on power moments of Kloosterman sums were
begun with the paper \cite{HS} of Sali\'{e} in 1931, where he
showed, for any odd prime $q$,
\begin{equation*}
MK^h=q^2M_{h-1}-(q-1)^{h-1}+2(-1)^{h-1}(h \geq 1).
\end{equation*}

Here $M_{0}=0$, and for $h \in \mathbb{Z}_{>0}$,

\begin{equation*}
M_{h}=|\{(\alpha_{1}, \cdots,\alpha_{h})\in(\mathbb{F}_{q}^{*})^{h}
| \sum_{j=1}^{h} \alpha_{j}=1=\sum_{j=1}^{h}\alpha_j^{-1}\}|.
\end{equation*}

For $q=p$ odd prime, Sali\'{e} obtained  $MK^{1}$, $MK^{2}$, $MK^{3}
$, $MK^{4}$ in \cite{HS} by determining $M_{1},M_{2},M_{3}$.
$MK^{5}$ can be expressed in terms of the $p$-th eigenvalue for a
weight 3 newform on $\Gamma_{0}$(15) (cf. \cite{RL}, \cite{CJM}).
$MK^{6}$ can be expressed in terms of the $p$-th eigenvalue for a
weight 4 newform on $\Gamma_{0}$(6) (cf. \cite{KJBD}). Also, based
on numerical evidence, in \cite{RJE}. Also, based on numerical
evidence, in \cite{RJE} Evans was led to propose a conjecture which
expresses $MK^7$ in terms of Hecke eigenvalues for a weight 3
newform on $\Gamma_0 (525)$ with quartic nebentypus of conductor
105. For more details about this brief history of explicit
computations on power moments of Kloosterman sums, one is referred
to Section IV of \cite{D2}.

From now on, let us assume that $q=2^{r}$. Carlitz\cite{L1}
evaluated $MK^{h}$ for $h \leq 4$. Recently, Moisio was able to find
explicit expressions of $MK^{h}$, for the other values of $h$ with
$h \leq 10$ (cf.\cite{M1}). This was done, via Pless power moment
identity, by connecting moments of Kloosterman sums and the
frequencies of weights in the binary Zetterberg code of length $q+1$
which were known by the work of Schoof and Vlugt in \cite{RM}.

In \cite{D2}, the binary linear codes $C(SL(n,q))$ associated with
finite special linear groups $SL(n,q)$ were constructed when $n,q$
are both powers of two. Then obtained was a recursive formula for
the power moments of multi-dimensional Kloosterman sums in terms of
the frequencies of weights in $C(SL(n,q)$. In particular, when
$n=2$, this gives a recursive formula for the power moments of
Kloosterman sums. Also, in order to get recursive formulas for the
power moments of Kloosterman and 2-dimensional Kloosterman sums, we
constructed in \cite{D3} three binary linear codes $C(SO^{+}(2,q))$,
$C(O^{+}(2,q) )$, $C(SO^{+}(4,q))$, respectively associated with
$SO^{+}(2,q)$, $O^{+} (2,q)$, $SO^{+}(4,q)$, and in \cite{D4} three
binary linear codes $C(SO^{-}(2,q))$, $C(O^{-}(2,q))$,
$C(SO^{-}(4,q))$, respectively associated with $SO^{-}(2,q)$,
$O^{-}(2,q)$, $SO^{-}(4,q)$. All of these were done via Pless power
moment identity and by utilizing our previous results on explicit
expressions of Gauss sums for the stated finite classical groups.
Still, in all, we had only a handful of recursive formulas
generating power moments of Kloosterman and 2-dimesional Kloosterman
sums.

In this paper, we will be able to produce infinite families of
recursive formulas generating power moments of Kloosterman and
2-dimensional Kloosterman sums. To do that, we construct two
infinite families of binary linear codes $C(DC^{-} (n,q))(n=1,3,5,
\cdots)$ and $C(DC^{+}(n,q))(n=2,4,6, \cdots)$, respectively
associated with the double cosets $DC^{-}(n,q)=P \sigma_{n-1}P$ and
$DC^{-}(n,q)=P \sigma_{n-2}P$ with respect to the maximal parabolic
subgroup $P=P(2n,q)$ of the symplectic group $Sp(2n,q)$, and express
those power moments in terms of the frequencies of weights in each
code. Then, thanks to our previous results on the explicit
expressions of exponential sums over those double cosets related to
the evaluations of ``Gauss sums" for the symplectic groups
$Sp(2n,q)$ \cite{D1}, we can express the weight of each codeword in
the duals of the codes in terms of Kloosterman or 2-dimensional
Kloosterman sums. Then our formulas will follow immediately from the
Pless power moment identity.

Theorem \ref{A} in the following(cf. (\ref{a5}), (\ref{a6}),
(\ref{a8})-(\ref{a10})) is the main result of this paper.
Henceforth, we agree that the binomial coefficient $\binom{b}{a}=0$,
if $a > b$ or $a<0$. To simplify notations, we introduce the
following ones which will be used throughout this paper at various
places.

\begin{equation}\label{a1}
A^{-}(n,q)=q^{ \frac{1}{4}(5n^2 -1)} \left[ \substack{n \\ 1}
 \right]_q\prod_{j=1}^{(n-1)/2}(q^{2j-1}-1),
\end{equation}
\begin{equation}\label{a2}
B^{-}(n,q)=q^{
\frac{1}{4}(n-1)^2}(q^n-1)\prod_{j=1}^{(n-1)/2}(q^{2j}-1),
\end{equation}
\begin{equation}\label{a3}
A^{+}(n,q)=q^{\frac{1}{4}(5n^2-2n)} \left[ \substack{n \\ 2}
 \right]_q
\prod_{j=1}^{(n-2)/2}(q^{2j-1}-1),
\end{equation}
\begin{equation}\label{a4}
B^{+}(n,q)=q^{\frac{1}{4}(n-2)^2}(q^n-1)(q^{n-1}-1)\prod_{j=1}^{(n-2)/2}(q^{2j}
-1).
\end{equation}
Henceforth we agree that the binomial coefficient $\binom{b}{a}=0$,
if $a>b$ and $a<0$.\\

\begin{theorem}\label{A}
 Let $q=2^{r}$. Then, with the notations in (\ref{a1})-(\ref{a4}), we have the following.
(a) For either each odd $n \geq 3$ and all $q$,  or $n=1$ and all $q
\geq 8$, we have a recursive formula generating power moments of
Kloosterman sums over $\mathbb{F}_{q}$

\begin{align}\label{a5}
  \begin{split}
 MK^h=& \sum_{l=0}^{h-1}(-1)^{h+l+1}\binom{h}{l}B^{-}(n,q)^{h-l}MK^{l}\\
      & +qA^{-}(n,q)^{-h}\sum_{j=0}^{min \{N^{-}(n,q),h\}}(-1)^{h+j}C_{j}^{-}(n,q)\\
      & \times \sum_{t=j}^{h}t!S(h,t)2^{h-t}\binom{N^-(n,q)-j}{N^{-}(n,q)-t}\\
      & (h=1,2,\cdots),
  \end{split}
\end{align}
where $N^{-}(n,q)=|DC^{-}(n,q)|=A^{-}(n,q)B^{-}(n,q)$, and
$\{C_{j}^{-}(n,q)\}_{j=0}^{N^{-}(n,q)}$ is the weight distribution
of $C(DC^{-}(n,q))$ given by

\begin{align}\label{a6}
  \begin{split}
 C_{j}^{-}(n,q)&= \sum_{}^{}\binom{q^{-1}A^-(n,q)(B^-(n,q)+1)}{\nu_{0}} \\
               &\times \prod_{tr(\beta^{-1} )=0} \binom{q^{-1}A^{-}(n,q)(B^{-}(n,q)+q+1)}{\nu_{\beta}} \\
               & \times \prod_{tr(\beta^{-1})=1}\binom{q^{-1}A^{-}(n,q)(B^{-}(n,q)-q+1)}{\nu_{\beta}}.
  \end{split}
\end{align}
Here the sum is over all the sets of nonnegative integers
$\{\nu_{\beta}\}_{\beta \in \mathbb{F}_{q}}$ satisfying $
\sum_{\beta \in \mathbb{F}_{q}} \nu_{\beta}=j$ and $\sum_{\beta \in
\mathbb{F}_{q}}^{} \nu_{\beta} \beta =0$. In addition, $S(h,t)$ is
the Stirling number of the second kind defined by
\begin{equation}\label{a7}
S(h,t)= \frac{1}{t!} \sum_{j=0}^{t}(-1)^{t-j}\binom{t}{j}j^{h}.
\end{equation}

(b) For each even $n \geq 2 $ and all $q \geq 4$, we have recursive
formulas generating power moments of 2-dimensional Kloosterman sums
over $\mathbb{F}_{q}$ and even power moments of Kloosterman sums
over $\mathbb{F}_{q}$

\begin{align}\label{a8}
  \begin{split}
 MK_2^h=&\sum_{l=0}^{h-1}(-1)^{h+l+1}\binom{h}{l}(B^{+}(n,q)-q^2)^{h-l}MK_2^l \\
        &+qA^{+}(n,q)^{-h}\sum_{j=0}^{min \{ N^{+}(n,q),h \}}(-1)^{h+j}C_{j}^{+}(n,q) \\
        & \times \sum_{t=j}^{h}t!S(h,t)2^{h-t} \binom{N ^{+}(n,q)-j}{N^{+}(n,q)-t}\\
        &( h=1,2,\cdots),
  \end{split}
\end{align}
and
\begin{align}\label{a9}
  \begin{split}
 MK^{2h}=&\sum_{l=0}^{h-1}(-1)^{h+l+1}\binom{h}{l}(B^+(n,q)-q^2+q)^{h-l}MK^{2 l}\\
         &+qA^{+}(n,q)^{-h} \sum_{j=0}^{min \{N^{+}(n,q),h \}}(-1)^{h+j}C_{j}^{+}(n,q)\\
         & \times \sum_{t=j}^{h}t!S(h,t)2^{h-t}\binom{N ^{+}(n,q)-j}{N^{+}(n,q)-t}\\
         & ( h=1,2,\cdots),
  \end{split}
\end{align}
where $N^+(n,q)=|DC^{+}(n,q)|=A^{+}(n,q)B^{+}(n,q)$, and $\{C_j^+
(n,q)\}_{j=0}^{N^+(n,q)}$ is the weight distribution of $C(DC^+
(n,q))$ given by

\begin{align}\label{a10}
  \begin{split}
  &C_{j}^{+}(n,q)\\
  &=\sum_{}^{}\binom{q^{-1}A^{+}(n,q)(B^{+}(n,q)+q^{3}-q^2-1)}{\nu_{0}} \prod_{\substack{|\tau |< 2\sqrt{q}\\ \tau \equiv -1(4)}}\\
  &\times\prod_{K(\lambda; \beta^{-1})=\tau}\binom{q^{-1}A^+(n,q)(B^+(n,q)+q \tau -q^2 -1)}{\nu_ {\beta}},
  \end{split}
\end{align}
Here the sum is over all the sets of nonnegative integers $\{\nu_
\beta\}_{\beta \in \mathbb{F}_{q}}$ satisfying  $\sum_{\beta \in
\mathbb{F}_{q}}\nu_{\beta}=j$, and $\sum_{ \beta \in \mathbb{F}_{q}
}\nu_{\beta} \beta=0$.\\

The following corollary is just $n=1$ the  case of (a) in the
above.\\
\end{theorem}

\begin{corollary}
Let $q \geq 8 $. For $h=1,2,\cdots $,
\begin{align*}
  \begin{split}
 MK^h=& \sum_{l=0}^{h-1}(-1)^{h+l+1}\binom{h}{l}(q-1)^{h-l}MK^l\\
      &+q^{1-h}\sum_{j=0}^{ min \{q(q -1),h \}}(-1)^{h+j}C_{j}^{-}(1,q)\\
      & \times \sum_{t=j}^{h}t!S(h,t)2^{h-t}\binom{q(q-1)-j}{q(q-1)-t}.
  \end{split}
\end{align*}
where $\{C_j^-(1,q)\}_{j=0}^{q(q-1)}$ is the weight distribution of
$C(DC^{-}(1,q))$ given by

\begin{equation*}
C_j^-(1,q)=\sum
\binom{q}{\nu_0}\prod_{tr(\beta^{-1})=0}\binom{2q}{\nu_{\beta}}.
\end{equation*}
Here the sum is over all the sets of nonnegative integers $\{\nu_0\}
\cup \{\nu_{\beta}\}_{tr(\beta^{-1})=0}$ satisfying $\nu_0+
\sum_{tr(\beta^{-1})=0}^{}\nu_ \beta=j$ and
$\sum_{tr(\beta^{-1})=0}\nu_ {\beta}\beta=0$. In addition, $S(h,t)$
is the Stirling number of the second kind given in (\ref{a7}).
\end{corollary}

%%%%%%%%%%%%%%%%%%%%%%%%%%%%%%%%%%%%%%%%%%%%%%%%%%%%%%%%%%%%%%%%%%%%
\section{$Sp(2n,q)$}
%%%%%%%%%%%%%%%%%%%%%%%%%%%%%%%%%%%%%%%%%%%%%%%%%%%%%%%%%%%%%%%%%%%%
For more details about this section, one is referred to the paper
\cite{D1}. Throughout this paper, the following notations will be
used:
\begin{itemize}
 \item [] $q = 2^r$ ($r \in \mathbb{Z}_{>0}$),\\
 \item [] $\mathbb{F}_{q}$ = the finite field with $q$ elements,\\
 \item [] $Tr A$ = the trace of $A$ for a square matrix $A$,\\
 \item [] $^tB$ = the transpose of $B$ for any matrix $B$.
\end{itemize}\
The symplectic group over the field  is defined as:

\begin{equation*}
Sp(2n,q)=\{w \in GL(2n,q)|{}^twJw=J \},
\end{equation*}
with
\begin{equation*}
J=\begin{bmatrix}
            0      & 1_{n} \\
            1_{n} & 0
  \end{bmatrix}
\end{equation*}

$P=P(2n,q)$ is the maximal parabolic subgroup of $Sp(2n,q)$ defined
by:
\begin{align*}
P(2n,q)&= \bigg\{
\begin{bmatrix}
   A & 0 \\
   0 & {}^{t}A^{-1}
\end{bmatrix}
\begin{bmatrix}
  1_{n} & B  \\
  0     & 1_{n}
\end{bmatrix}
\bigg| A \in GL(n,q), {}^{t}B=B\}
\end{align*}

Then, with respect to $P=P(2n,q)$, the Bruhat decomposition of
$Sp(2n,q)$ is given by
\begin{equation}\label{a11}
Sp(2n,q)=\coprod_{r=0}^{n} P\sigma_{r}P,
\end{equation}
where
\[
\sigma_r=
\begin{bmatrix}
   0   & 0         & 1_r  & 0 \\
   0   & 1_{n-r}  & 0    & 0 \\
  1_{r} & 0       & 0  &0\\
   0   & 0         & 0    & 1_{n-r}
\end{bmatrix}
\in Sp(2n,q).
\]

Put, for each $r$ with $0 \leq r \leq n$,
\begin{equation*}
A_r=\{w \in P(2n,q)| \sigma_{r}w \sigma_{r}^{-1} \in P(2n,q)\}.
\end{equation*}

Expressing $Sp(2n,q)$ as a disjoint union of right cosets of
$P=P(2n,q)$, the Bruhat decomposition in (11) can be written as

\begin{equation*}
Sp(2n,q)=\coprod_{r=0}^{n}P \sigma_{r}(A_{r}\backslash P).
\end{equation*}

The order of the general linear group $ GL(n,q)$ is given by

\begin{equation*}
g_n=\prod_{j=0}^{n-1}(q^n-q^j)=q^{\binom{n}{2}}\prod_{j=1}^{n}(q^{j}-1).
\end{equation*}

For integers $n,r$ with $0 \leq r \leq n$, the $q$-binomial
coefficients are defined as:

\begin{equation*}
\left[ \substack{n \\ r}
 \right]_q = \prod_{j=0}^{r-1} (q^{n-j} - 1)/(q^{r-j} - 1).
 \end{equation*}

Then, for integers $n,r$  with $0 \leq r \leq n$, we have
\begin{equation}\label{a12}
\frac{g_n}{g_{n-r} g_r} = q^{r(n-r)}\left[ \substack{n \\ r}
 \right]_q.
\end{equation}

In \cite{D1}, it is shown that
\begin{equation}\label{a13}
|A_{r}|=g_{r}g_{n-r}q^{\binom{n+1}{2}}q^{r(2n-3r-1)/2}.
\end{equation}
Also, it is immediate to see that
\begin{equation}\label{a14}
|P(2n,q)|=q^{\binom{n+1}{2}}g_n.
\end{equation}
So, from (\ref{a12})-(\ref{a14}), we get

\begin{equation}\label{a15}
\mid A_r\backslash P(2n,q) \mid = q^{\binom{r+1}{2}}\left[ \substack{n\\
r}
 \right]_q,
\end{equation}
and
\begin{align}\label{a16}
\begin{split}
  &\mid P(2n,q)\sigma_{r}P(2n,q)\mid\\
  &= \mid P(2n,q) \mid^{2} \mid A_r \mid^{-1}\\
  & =q^{n^{2}}\left[ \substack{n\\r} \right]_q q^{{r \choose 2}}q^{r}\prod_{j=1}^{n}(q^{j}-1).
\end{split}
\end{align}

In particular, with
\begin{align*}
 \begin{split}
&DC^-(n,q)=P(2n,q)\sigma_{n-1}P(2n,q),\\
&DC^+(n,q)=P(2n,q) \sigma_{n-2}P(2n,q),
 \end{split}
\end{align*}

\begin{equation}\label{a17}
|DC^-(n,q)|=q^{\frac{1}{2 }n(3n-1)} \left[ \substack{n\\1} \right]_q
\prod_{j=1}^n(q^{j}-1),
\end{equation}
\begin{equation}\label{a18}
|DC^+(n,q)|=q^{\frac{1}{2}(3n^2-3n+2)}\left[ \substack{n\\2}
\right]_q \prod_{j=1}^n(q^j-1).
\end{equation}

Also, from (\ref{a11}), (\ref{a16}), we have
\begin{align*}
  \begin{split}
|Sp(2n,q)|&=\sum_{r=0}^{n}|P(2n,q)|^2|A_r |^{-1}\\
          &=q^{n^2}\prod_{j=1}^n(q^{2j}-1),
  \end{split}
\end{align*}
where one can apply the following $q$-binomial theorem with $x=-q$:
\begin{equation*}
\sum_{r=0}^{n}\left[ \substack{n\\r} \right]_q(-1)^{r}
q^{\binom{r}{2}} x^r=(x; q)_{n},
\end{equation*}
with $(x;q)_n=(1-x)(1-qx)\cdots (1-q^{n-1}x)$ ($x$ an indeterminate,
$n \in \mathbb{Z}_{>0}$).

%%%%%%%%%%%%%%%%%%%%%%%%%%%%%%%%%%%%%%%%%%%%%%%%%%%%%%%%%%%%%%%%%%%%
\section{Exponential sums over double cosets of $Sp(2n,q)$}
%%%%%%%%%%%%%%%%%%%%%%%%%%%%%%%%%%%%%%%%%%%%%%%%%%%%%%%%%%%%%%%%%%%%

The following notations will be used throughout this paper.
\begin{gather*}
tr(x)=x+x^2+\cdots+x^{2^{r-1}} \text{the trace function}
~\mathbb{F}_{q}
\rightarrow \mathbb{F}_2,\\
\lambda(x) = (-1)^{tr(x)} ~\text{the canonical additive character
of} ~\fd_q.
\end{gather*}
Then any nontrivial additive character $\psi$ of $\fd_q$ is given by
$\psi(x) = \lambda(ax)$ , for a unique $a \in \fd_q^*$.

For any nontrivial additive character $\psi$ of $\fd_q$ and $a \in
\fd_q^*$, the Kloosterman sum $K_{GL(t,q)}(\psi ; a)$ for $GL(t,q)$
is defined as
\begin{equation*}
K_{GL(t,q)}(\psi ; a) = \sum_{w \in GL(t,q)} \psi(Trw + a~Trw^{-1}).
\end{equation*}
Notice that, for $t=1 $, $ K_{GL(1,q)}( \psi;a)$ denotes the
Kloosterman sum $K(\psi;a) $.

For the Kloosterman sum $K(\psi;a)$, we have the Weil bound(cf.
\cite{RH})
\begin{equation}\label{a19}
\mid K(\psi ; a) \mid \leq 2\sqrt{q}.
\end{equation}

In \cite{D1}, it is shown that $K_{GL(t,q)}(\psi ; a)$ ~satisfies
the following recursive relation: for integers $t \geq 2$, ~$a \in
\fd_q^*$ ,
\begin{multline}\label{a20}
K_{GL(t,q)}(\psi ; a) = q^{t-1}K_{GL(t-1,q)}(\psi ; a)K(\psi
;a)\\
+ q^{2t-2}(q^{t-1}-1)K_{GL(t-2,q)}(\psi ; a),
\end{multline}
where we understand that $K_{GL(0,q)}(\psi ; a)=1$ . From
(\ref{a20}), in \cite{D1} an explicit expression of the Kloosterman
sum for
$GL(t,q)$ was derived.\\

\begin{theorem}[\cite{D1}]\label{C} For integers $t \geq 1$, and $a \in \fd_q^*$, the
Kloosterman sum $K_{GL(t,q)}(\psi ; a)$ is given by
\begin{align*}
\begin{split}
 K_{GL(t,q)}(\psi ; a)=q^{(t-2)(t+1)/2} &\sum_{l=1}^{[(t+2)/2]} q^l K(\psi;a)^{t+2-2l}\\
                                         & \times \sum \prod_{\nu=1}^{l-1} (q^{j_\nu -2\nu}-1),
\end{split}
\end{align*}
where  $K(\psi;a)$ is the Kloosterman sum and the inner sum is over
all integers $j_1,\ldots,j_{l-1}$ satisfying $2l-1 \leq j_{l-1} \leq
j_{l-2} \leq \cdots \leq j_1 \leq t+1$. Here we agree that the inner
sum is $1$ for $l=1$.
\end{theorem}

In Section 5 of \cite{D1}, it is shown that the Gauss sum for
$Sp(2n,q)$ is given by:

\begin{align}\label{a21}
\begin{split}
  &\sum_{w \in Sp(2n,q)} \psi(Tr w)\\
  &=\sum_{r=0}^{n} \sum _{ w \in P \sigma_{r}P}^{} \psi(Tr w)\\
  &=\sum_{r=0}^{n}|A_{r} \backslash P| \sum_{w \in P } \psi(Tr w \sigma_{r})\\
  &=q^{\binom{n+1}{2}}\sum_{r=0}^{n} |A_{r} \backslash P | q^{r(n-r)}a_{r}K_{GL(n-r,q)}( \psi;1).
\end{split}
\end{align}

Here $\psi$ is any nontrivial additive character of $\fd_{q}$,
$a_{0}=1$, and, for $r \in \z_{>0}$, $a_r$ denotes the number of all
$r \times r$ nonsingular alternating matrices over $\fd_q$, which is
given by
\begin{equation}\label{a22}
a_{r}= \begin{cases}
    0, & \hbox{if $r$ is odd,} \\
    q^{\frac{r}{2}(\frac{r}{2}-1)} \prod_{j=1}^{\frac{r}{2}}(q^{2j-1}), & \hbox{if $r$ is even,} \\
\end{cases}
\end{equation}
(cf. \cite{D1},  Proposition 5.1).

Thus we see from (\ref{a15}), (\ref{a21}), and (\ref{a22}) that, for
each $ r$ with $0 \leq r \leq n$,
\begin{equation}\label{a23}
\sum_{w \in P \sigma_{r}P}^{}\psi(Trw)=
\begin{cases}
0, & \hbox{if $r$ is odd,} \\
\begin{split}&q^{\binom{n+1}{2}}q^{rn-\frac{1}{4}r^2}\left[ \substack{n\\r} \right]_q \\
             &\times \prod_{j=1}^{r/2}(q^{2j-1}-1)\\
             &\times K_{GL(n-r,q)}(\psi;1)\end{split}, & \hbox{if $r$ is even.} \\
\end{cases}
\end{equation}

For our purposes, we need two infinite families of exponential sums
in (\ref{a23}) over $P(2n,q) \sigma_{n-1}P(2n,q)=DC^-(n,q)$ for
$n=1,3,5,\cdots$ and over $P(2n,q) \sigma_{n-2}P(2n,q)=DC^+(n,q)$
for $n=2,4,6,\cdots$. So we state them separately as a theorem.\\

\begin{theorem}
 Let $\psi$  be any nontrivial additive character of $\fd_q$. Then, in the notations of (1) and (3), we have
\begin{equation*}
\sum_{w \in DC^{-}(n,q)} \psi(Trw)=A^-(n,q)K(\psi;1), \textmd{for}
\;\; n=1,3,5,\cdots,
\end{equation*}

\begin{align*}
  \begin{split}
   &\sum_{w \in DC^+(n,q)}\psi(Trw)=q^{-1}A^+(n,q)K_{GL(2,q)}(\psi;1)\\
    & =A^+(n,q)(K(\psi;1)^2+q^2-q),\textmd{for} \;\;n=2,4,6,\cdots
  \end{split}
\end{align*}
(cf. (\ref{a23}), (\ref{a20})).\\
\end{theorem}

\begin{proposition}[\cite{D3}]\label{E}
For $n=2^s(s \in \mathbb{Z}_{\geq 0})$, and $\psi$  a nontrivial
additive character of $\fd_q$,
\[
K(\psi;a^n) = K(\psi;a).
\]
\end{proposition}
We need a result of Carlitz for the next corollary.\\

\begin{theorem}[\cite{L2}]\label{F}
For the canonical additive character $\lambda$ of $\fd_q$, and $a
\in \fd_{q} ^{*}$,
\begin{equation}\label{a24}
K_{2}(\lambda;a) = K(\lambda;a)^{2}-q.
\end{equation}
\end{theorem}

The next corollary follows from Theorem 4, Proposition 5,
(\ref{a24}), and
simple change of variables.\\

\begin{corollary}\label{G}
Let $\lambda $ be the canonical additive character of $\fd_{q}$, and
let $a \in \fd_{q}^{*}$. Then we have

\begin{align}\label{a25}
  \begin{split}
\sum_{w \in DC^{-}(n,q)}\lambda(aTr w)=A^-(n,q)&K(\lambda;a),\\
                                                 &\textmd{for} \;\;n=1,3,5, \cdots,
  \end{split}
\end{align}

\begin{align}\label{a26}
  \begin{split}
 \sum_{w \in DC^{+}(n,q)} \lambda(aTrw)&=A^+(n,q)(K(\lambda ;a)^2+q^2-q)\\
                                      &=A^+(n,q)(K_2(\lambda ;a)+q^2),\\
                                      &\textmd{for}\;\;n=2,4,6,\cdots
  \end{split}
\end{align}
(cf. (\ref{a1}), (\ref{a3})).\\
\end{corollary}

\begin{proposition}[\cite{D3}]\label{H}
Let $\lambda$ be the canonical additive character of $ \fd_{q}$, $ m
\in \z_{>0}$, $\beta \in \fd_{q}$. Then

\begin{align}\label{a27}
  \begin{split}
 &\sum_{a \in \fd_{q}^{*}} \lambda(-a \beta)K_m (\lambda ;a)\\
 &=\begin{cases}
    qK_{m-1}(\lambda ;\beta^{-1})+(-1)^{m+1}, & \hbox{if $\beta \neq 0$,} \\
      (-1)^{m+1}, & \hbox{if $\beta = 0$,} \\
\end{cases}
\end{split}
\end{align}
with the convention $K_{0}(\lambda ; \beta^{-1})= \lambda(\beta
^{-1})$.
\end{proposition}

For any integer $r$ with $ 0 \leq r \leq n$, and each $\beta \in
\fd_{q}$, we let
\begin{equation*}
N_{P \sigma_{r} P}(\beta)=|\{w \in P \sigma_r P |Tr w= \beta \}|.
\end{equation*}
Then it is easy to see that

\begin{equation}\label{a28}
qN_{P\sigma_{r}P}(\beta)=|P \sigma_{r}P|+\sum_{a \in \fd_{q}^{*}}
\lambda(-a \beta) \sum_{w \in P \sigma_r P }\lambda(a Tr w).
\end{equation}
Now, from (\ref{a25})-(\ref{a28}), (\ref{a17}), and (\ref{a18}), we
have the following result.\\

\begin{proposition}\label{I}
(a) For $n=1,3,5, \cdots $
\begin{align}\label{a29}
  \begin{split}
&N_{DC^-(n,q)}(\beta)\\
&=q^{-1}A^-(n,q)B^-(n,q)+q^{-1}A^-(n,q)\\
& \times \begin{cases} 1,& \beta=0,\\
q+1, & tr(\beta^{-1})=0,\\
-q+1,&tr( \beta^{-1})=1.
\end{cases}
  \end{split}
\end{align}

(b) For $n=2,4,6, \cdots,$
\begin{align}\label{a30}
  \begin{split}
 &N_{DC^{+}(n,q)}(\beta) \\
 &=q^{-1}A^{+}(n,q)B^{+}(n,q)+q^{-1}A^{+}(n,q)\\
 & \times \begin{cases} qK(\lambda ; \beta^{-1})-q^{2}-1,& \beta \neq 0,\\
 q^3-q^2-1,& \beta=0.
 \end{cases}
  \end{split}
\end{align}
(cf. (\ref{a1})- (\ref{a4})).\\
\end{proposition}

\begin{corollary}\label{J}
(a) For all odd $n \geq 3$ and all $q$, $N_{DC^{-}(n,q)}(\beta)> 0$,
for all $\beta$;for $n=1 $ and all $q$,

\begin{equation}\label{a31}
N_{DC^{-}(1,q)}(\beta)
 =\begin{cases}
    q, & \beta =0 ,\\
     2q , & tr(\beta^{-1})=0, \\
      0, & tr(\beta^{-1}=1.
\end{cases}
\end{equation}

(b) For all even $ n \geq 4$ and all $ q$, or $ n=2$ and all $ q
\geq 4$, $N_{DC^{+}(n,q)}(\beta)> 0$, for all $\beta$;for $n=2$ and
$q=2$

\begin{equation*}\label{}
N_{DC^{+}(2,2)}(\beta)
 =\begin{cases}
   0 , & \beta =1,\\
    48=|P(4,2)| ,& \beta=0.
\end{cases}
\end{equation*}
\end{corollary}

\begin{proof}
(a) $n=1 $ case follows directly from (\ref{a29}). Let $ n \geq 3$
be odd. Then, from (\ref{a29}), we see that, for any $\beta$,

\begin{align*}
  \begin{split}
 &N_{DC^{-}(n,q)}(\beta)\\
 &\geq q^{\frac{1}{2}(3n^2 -n-2)}(q^n-1)\prod_{j=2}^n(q^j-1)-q^{\frac{5}{4}(n^2-1)}\\
                                                                      & \times \prod_{j=1}^{(n+1)/2}(q^{2j-1}-1)\\
                        &> q^{\frac{1}{2}(3n^2-n-2)}(q^n -1 )\prod_{j=2}^n(q^j-1)-q^{\frac{5}{4}(n^2-1)}\prod_{j=1}^{(n+1)/2}q^{2j-1}\\
                        &= q^{\frac{1}{2}(3n^2-n-2)}\{(q^n -1)(\prod_{j=2}^n(q^j-1)-1)-1 \} >0.
  \end{split}
\end{align*}

(b) Let $n=2$. Let $\beta \neq 0$. Then, from (\ref{a30}), we have
\begin{equation}\label{a32}
N_{DC^+(2,q)}(\beta)=q^4 \{q^2-2q-1 + K(\lambda ; \beta^{-1})\},
\end{equation}
where $q^2-2q-1+ K(\lambda; \beta^{-1}) \geq q^2-2q-1-2 \sqrt{q}>0$,
for $q \geq 4$, by invoking the Weil bound in (\ref{a19}). Also,
observe from (\ref{a32}) that $N_{DC^+(2,2)}(1)=0$.

On the other hand, if $\beta =0 $, then, from (\ref{a30}), we get
\begin{equation*}
N_{DC^{+}(2,q)}(0)=q^4(2q^2-2q-1 )>0, \textmd{for all}\;\;q \geq 2.
\end{equation*}
In addition, we note that $ N_{DC^{+}(2,2)}(0)=48$.

Assume now that $n \geq 4 $. If $\beta=0$, then, from (\ref{a30}),
we see that $N_{DC^{+}(n,q)}(0)>0$, for all $q$. Let $\beta \neq 0$.
Then, again by invoking the Weil bound,
\begin{align*}
  \begin{split}
 &N_{DC^+ (n,q)}(\beta) \geq q^{-1}A^+(n,q)\\
 & \times \{(q^n-1)(q^{n-1}-1)q^{\frac{n^2}{4}-n +1}\\
  &\times \prod_{j=1}^{(n-2)/2}(q^{2j}-1)-(q^2 +2q^{\frac{3}{2}}+1)\}.
  \end{split}
\end{align*}

Clearly, $\prod_{j=1}^{(n-2)/2} (q^{2j}-1)>1$. So we only need to
show, for all $ q \geq 2$,

\begin{align*}
  \begin{split}
 f(q)=&(q^n-1)(q^{n-1}-1)q^{\frac{n^2}{4}-n+1}\\
      &-(q^2+2q^{\frac{3}{2}}+1)>0.
  \end{split}
\end{align*}
But, as $ n \geq 4 $, $f(q) \geq q(q^4
-1)(q^3-1)-(q^2+2q^{\frac{3}{2}}+1)>0$, for all $q \geq 2$.
\end{proof}
%%%%%%%%%%%%%%%%%%%%%%%%%%%%%%%%%%%%%%%%%%%%%%%%%%%%%%%%%%%%%%%%%%%%
\section{Construction of codes}
%%%%%%%%%%%%%%%%%%%%%%%%%%%%%%%%%%%%%%%%%%%%%%%%%%%%%%%%%%%%%%%%%%%%
Let
\begin{align}\label{a33}
  \begin{split}
N^{-}(n,q)=&|DC^{-}(n,q)|=A^{-}(n,q)B^{-}(n,q),\\
            &\textmd{for}\;\; n=1,3,5, \cdots,
  \end{split}
\end{align}

\begin{align}\label{a34}
  \begin{split}
N^{+}(n,q)=&|DC^{+}(n,q)|=A^{+}(n,q)B^{+}(n,q),\\
            &\textmd{for}\;\; n=2,4,6, \cdots
  \end{split}
\end{align}
(cf. (\ref{a17}), (\ref{a18}), (\ref{a1})-(\ref{a4})).

Here we will construct two infinite families of binary linear codes
$C(DC^{-}(n,q))$ of length $N^-(n,q)$ for all positive odd integers
$n$ and all $q$, and $C(DC^{+}(n,q))$ of length $N^{+}(n,q)$ for all
positive even integers $n$ and all $q$, respectively associated with
the double cosets $DC^{-}(n,q)$ and $DC^{+}(n,q)$.

Let $g_{1}, g_{2}, \cdots, g_{N^{-}(n,q)}$ and $
g_{1},g_{2},\cdots,g_{N^{+}(n,q)}$ be respectively fixed orderings
of the elements in $DC^{-}(n,q) (n=1,3,5,\cdots)$ and $DC^+
(n,q)(n=2,4,6,\cdots )$, by abuse of notations. Then we put
\begin{align*}
  \begin{split}
 v^{-}(n,q)=(Trg_{1},Trg_{2},& \cdots,Trg_{N^{-}(n,q)}) \in \fd_q^{N^-(n,q)},\\
                             &\textmd{for} \;\;n=1,3,5, \cdots,
 \end{split}
\end{align*}

\begin{align*}
  \begin{split}
 v^{+}(n,q)=(Trg_{1},Trg_{2},& \cdots,Trg_{N^+(n,q)}) \in \fd_q^{N^+(n,q)},\\
                              &\textmd{for}\;\; n=2,4,6,\cdots.
  \end{split}
\end{align*}

Now, the binary codes $C(DC^-(n,q))$ and $ C(DC^+(n,q))$ are defined
as:
\begin{align}\label{a35}
  \begin{split}
 C(DC^-(n,q))=\{u \in \fd_2^{ N^-(n,q)}|& u \cdot v^-(n,q)=0 \},\\
                                        &\textmd{for}\;\; n=1,3,5, \cdots ,
  \end{split}
\end{align}

\begin{align}\label{a36}
  \begin{split}
 C(DC^+(n,q))=\{u \in \fd_2^{N^+(n,q)}|& u \cdot v^+(n,q)=0\},\\
                                       &\textmd{for}\;\; n=2,4,6, \cdots,
  \end{split}
\end{align}
where the dot denotes the usual inner product in $\fd_q^{N^-(n,q)}$
and $\fd_q^{N^+(n,q)}$, respectively.

The following Delsarte's theorem is well-known.\\

\begin{theorem}[\cite{FN}]\label{K}
Let $B$ be a linear code over $ \fd_{q}$. Then
\begin{equation*}
(B|_{\fd_{2}})^{\bot }=tr(B^{\bot}).
\end{equation*}
\end{theorem}

In view of this theorem, the duals  $C(DC^-(n,q))^{\bot}$ and
$C(DC^+(n,q))^{\bot}$ of the respective codes $C(DC^-(n,q))$ and
$C(DC^+(n,q))$ are given by

\begin{align}\label{a37}
  \begin{split}
 &C(DC^{-}(n,q))^{\bot} \\
 &= \{c^{-}(a)=c^{-}(a;n,q)\\
 &\qquad\quad\quad=(tr(aTrg_{1}), \cdots,tr(aTrg_{N^{-}(n,q)}))|a \in \fd_q \}\\
 &(n=1,3,5,\cdots),
  \end{split}
\end{align}

\begin{align}\label{a38}
  \begin{split}
 &C(DC^{+}(n,q))^{\bot}\\
  &=\{c^{+}(a)=c^{+}(a;n,q)\\
  &\qquad\quad\quad=(tr(aTrg_{1}),\cdots,tr(aTrg_{N^{+}(n,q)})) |a \in \fd_{q}\}\\
  &(n=2,4,6,\cdots).
  \end{split}
\end{align}

Let  $\fd_2^+,\fd_q^+$ denote the additive groups of the fields
$\fd_2,\fd_q$, respectively. Then we have the following exact
sequence of groups:
\begin{equation*}
0 \rightarrow \fd_2^+ \rightarrow \fd_q^+ \rightarrow \Theta(\fd_q)
\rightarrow 0,
\end{equation*}
where the first map is the inclusion and the second one is  the
Artin-Schreier operator in characteristic two given by $x \mapsto
\Theta(x) = x^2+x$. So
\begin{equation}\label{a39}
\Theta(\fd_q) = \{\alpha^2 + \alpha \mid  \alpha \in \fd_q \},~
\textmd{and} ~~[\fd_q^+ : \Theta(\fd_q)] = 2.
\end{equation}

\begin{theorem}[\cite{D3}]\label{L}
Let $\lambda$  be the canonical additive character of $\fd_q$, and
let $\beta \in \fd_q^*$. Then
\begin{equation}\label{a40}
 (a) \sum_{\alpha \in
 \fd_q-\{0,1\}}\lambda(\frac{\beta}{\alpha^2+\alpha})=K(\lambda;\beta)-1,
 \qquad \qquad \qquad \qquad \qquad \qquad \qquad \qquad
\end{equation}
\begin{equation*}
(b)\sum_{\alpha \in
\fd_q}\lambda(\frac{\beta}{\alpha^2+\alpha+b})=-K(\lambda;\beta)-1,
\qquad \qquad \qquad \qquad \qquad \qquad \qquad
\end{equation*}
if $x^2+x+b (b\in \fd_q)$ is irreducible over $\fd_q$, or
equivalently if $b \in \fd_q\setminus\Theta(\fd_q)$ (cf.
\;(\ref{a39})).\\
\end{theorem}

\begin{theorem}\label{M}
(a) The map $\fd_q \rightarrow C(DC^-(n,q))^{\bot}(a \mapsto
c^-(a))$ is an $\fd_{2}$-linear isomorphism for $n \geq 3 $ odd and
all $q $, or $n=1 $ and $q \geq 8$.

(b) The map $\fd_q \rightarrow C(DC^+(n,q))^{\bot}(a \mapsto
c^+(a))$ is an $\fd_{2}$-linear isomorphism for $n \geq 4 $ even and
all $q$, or $n=2$ and $q \geq 4$.
\end{theorem}
\begin{proof}
(a) The map is clearly $\fd_{2}$-linear and surjective. Let $a$ be
in the kernel of map. Then $tr(aTrg)=0$, for all $g \in DC^-(n,q)$.
If $n \geq 3$ is odd, then, by Corollary 10 (a), $Tr : DC^-(n,q)
\rightarrow \fd_q$ is surjective and hence $tr(a \alpha)=0$, for all
$\alpha \in \fd_q$. This implies that $a=0$, since otherwise $tr
:\fd_q \rightarrow \fd_2$ would be the zero map. Now, assume that $
n=1$ and $q \geq 8$. Then, by (31), $tr (a \beta)=0$, for all $\beta
\in \fd_{q}^{*}$, with $tr(\beta^{-1})=0$. Hilbert's theorem 90 says
that $tr( \gamma)=0 \Leftrightarrow \gamma=\alpha^2 + \alpha$, for
some $\alpha \in \fd_q$. This implies that $\sum_{\alpha \in \fd_q -
\{0,1\}} \lambda (\frac{a}{\alpha^2+\alpha})=q-2$. If $a \neq 0$,
then, using (40) and the Weil bound (19), we would have
\begin{equation*}
q-2=\sum_{\alpha \in \fd_q - \{0,1\}}  \lambda(\frac{a}{\alpha^2+
\alpha})=K(\lambda;a)-1 \leq 2 \sqrt{q}-1.
\end{equation*}
But this is impossible, since $x > 2 \sqrt{x} +1$, for $x \geq 8$.

(b) This can be proved in exactly the same manner as in the $n \geq
3$ odd case of (a) (cf. Corollary 10 (b)).\\
\end{proof}

Remark: One can show that the kernel of the map $\fd_q \rightarrow
C(DC^-(1,q))^{\bot}(a \mapsto c^-(a))$, for $q=2,4$ and of the map $
\fd_q \rightarrow C(DC^+(2,2))^{\bot}(a \mapsto c^+(a))$ are all
equal to $\fd_2$.

%%%%%%%%%%%%%%%%%%%%%%%%%%%%%%%%%%%%%%%%%%%%%%%%%%%%%%%%%%%%%%%%%%%%%%%
\section{Recursive formulas for power moments of Kloosterman sums}
%%%%%%%%%%%%%%%%%%%%%%%%%%%%%%%%%%%%%%%%%%%%%%%%%%%%%%%%%%%%%%%%%%%%%%%
Here we will be able to find, via Pless power moment identity,
infinite families of recursive formulas generating power moments of
Kloosterman and 2-dimensional Kloosterman sums over all
$\fd_{q}$(with three exceptions) in terms of the frequencies of
weights in $ C(DC^-(n,q))$ and $C(DC^+(n,q))$, respectively.

\begin{theorem}\label{N}(Pless power moment identity, \cite{FN}):
Let $ B$ be an $q$-ary $[n,k]$ code, and let $B_{i}$(resp.$B_{i}
^{\bot})$ denote the number of codewords of weight $i$ in $B$(resp.
in $B^{\bot})$. Then, for $h=0,1,2, \cdots$,

\begin{align}\label{a41}
  \begin{split}
&\sum_{j=0}^{n}j^{h}B_{j}\\
&=\sum_{j=0}^{min \{ n,h \}}(-1)^{j}B_{j} ^{\bot} \sum_{t=j}^{h} t!
S(h,t)q^{k-t}(q-1)^{t-j}\binom{n-j}{n-t},
  \end{split}
\end{align}
\end{theorem} where $S(h,t)$ is the Stirling number of the second
kind defined in (\ref{a7}).\\

\begin{lemma}
Let
\begin{equation*}
c^-(a)=(tr(a Trg_{1}), \cdots, tr(a Trg_{N^-(n,q)})) \in
C(DC^-(n,q))^{\bot}
\end{equation*}$(n=1,3,5,\cdots)$, and let
\begin{equation*}
c^+(a)=(tr(a Tr g_{1}), \cdots,tr(a Trg_{N^+(n,q)})) \in
C(DC^+(n,q))^{\bot}
\end{equation*}$(n=2,4,6,\cdots)$,
for $ a \in \fd_{q}^{*}$. Then the Hamming weights $w(c^-(a))$ and
$w(c^+(a))$ are expressed as follows:

\begin{equation}\label{a42}
(a)w(c^-(a))= \frac{1}{2}A^-(n,q)(B^-(n,q)-K(\lambda ;
a)).\quad\quad
\end{equation}
\begin{equation}\label{a43}
(b)w(c^{+}(a))=\frac{1}{2}A^+(n,q)(B^+(n,q)-q^2+q-K(\lambda ;a)^2)
\end{equation}
\begin{equation}\label{a44}
=\frac{1}{2}A^+(n,q)(B^+(n,q)-q^2-K_2(\lambda ; a))
\end{equation}
(cf. (1)- (4)).\\

\begin{proof}
$w (c^{\mp}(a))=\frac{1}{2} \sum_{j=1}^{N^{\mp}(n,q)}(1-(-1)^{tr(aTr
g_j)})=\frac{1}{2}(N^{\mp}(n,q)-\sum_{ w \in DC^{\mp}(n,q)}
\lambda(a Tr w))$. Our results now follow from (\ref{a33}),
(\ref{a34}), (\ref{a25}), and
(\ref{a26}).\\
\end{proof}
\end{lemma}

Let $u=(u_1, \cdots, u_{N_{N^{\mp}(n,q)}}) \in
\fd_2^{N^{\mp}(n,q)}$, with $\nu_{\beta}$ 1's in the coordinate
places where $Tr(g_j)= \beta$, for each $\beta \in \fd_q$. Then from
the definition of the codes $C(DC^{\mp}(n,q))$ (cf. (35), (36)) that
$u$ is a codeword with weight $j$ if and only if $\sum_{\beta \in
\fd_{q}}^{} \nu _{\beta}=j$ and $\sum_{\beta \in \fd_{q}}^{}
\nu_{\beta} \beta=0$ (an identity in $\fd_{q}$). As there are
$\prod_{\beta \in \fd_{q}}
\binom{N_{DC^{\mp}(n,q)}(\beta)}{\nu_{\beta}}$ many such codewords
with weight $j$, we obtain the following result.\\

\begin{proposition}
Let $\{C_{j}^{-}(n,q)\}_{j=0}^{N^-(n,q)}$ be the weight distribution
of $C(DC^{-}(n,q))$ $(n=1,3,5, \cdots)$, and let
$\{C_{j}^{+}(n,q)\}_{j=0}^{N^+(n,q)}$ be that of $C(DC^+(n,q))$
$(n=2,4,6,\cdots)$. Then
\begin{equation}\label{a45}
C_{j}^{\mp}(n,q)=\sum \prod_{\beta \in
\fd_q}\binom{N_{DC^{\mp}(n,q)}(\beta)}{\nu_{\beta}},
\end{equation}
where the sum is over all the sets of integers $\{\nu_{\beta }\}_{
\beta \in \fd_q}(0 \leq \nu_{\beta} \leq N_{DC^{\mp}(n,q)}(\beta)
)$, satisfying
\begin{equation}\label{a46}
\sum_{\beta \in \fd_{q}}^{} \nu_{\beta}=j,\;\; \textmd{and}
\sum_{\beta \in \fd_{q}}^{} \nu_{\beta} \beta=0.
\end{equation}
\end{proposition}

\begin{corollary}\label{Q}
Let $\{C_{j}^{-}(n,q)\}_{j=0}^{N^-(n,q)}$ $(n=1,3,5, \cdots)$, $\{
C_{j}^{+}(n,q)\}_{j=0}^{N^{+}(n,q)}$ $(n=2,4,6, \cdots)$ be as
above. Then we have
\begin{align*}
  \begin{split}
&C_{j}^{\mp}(n,q)=C_{N^{\mp}(n,q)-j}^{\mp}(n,q),\\
 &\textmd{for all}\;\; j,\;\; \textmd{with}\;\; 0 \leq j \leq N^{\mp}(n,q).
  \end{split}
\end{align*}
\end{corollary}

\begin{proof}
Under the replacements $\nu_{\beta}\rightarrow N_{DC^{\mp}(n,q)}(
\beta)- \nu_{\beta}$, for each  $\beta \in \fd_q$, the first
equation in (\ref{a46}) is changed to $N^{\mp}(n,q)-j$, while the
second one in there and the summands in (\ref{a45}) are left
unchanged. Here the second sum in (\ref{a46}) is left unchanged,
since $\sum_{\beta \in \fd_q} N_{DC^{\mp}(n,q)}(\beta) \beta=0$, as
one can see by using the explicit expressions of
$N_{DC^{\mp}(n,q)}(\beta)$ in (\ref{a29}) and (\ref{a30}).\\
\end{proof}

\begin{theorem}[\cite{GJ}]\label{R}
Let $q=2^r$, with $ r \geq 2$. Then the range $R$ of $K(\lambda ;a)
$, as $a$ varies over $\fd_q^{*}$, is given by:
\begin{equation*}
R=\{\tau \in \mathbb{Z} \; | \; |\tau |<2 \sqrt{q}, \; \tau \equiv
-1 (mod \; 4) \}.
\end{equation*}
In addition, each value $\tau \in R $ is attained exactly $H(\tau^2
-q)$ times, where $H(d)$ is the Kronecker class number of $d$.
\end{theorem}

The formulas appearing in the next theorem and stated in (\ref{a6})
and (\ref{a10}) follow  by applying the formula in (\ref{a45}) to
each $C(DC^{\mp}(n,q))$, using the explicit values of $N_{DC^{\mp}
(n,q)}(\beta)$ in (\ref{a29}) and (\ref{a30}), and taking Theorem
\ref{R}
into consideration.\\

\begin{theorem}\label{S}
Let $\{C_{j}^{-}(n,q)\}_{j=0}^{N^-(n,q)}$ be the weight distribution
of  $C(DC^-(n,q))$ $(n=1,3,5,\cdots)$, and let $\{C_{j}^{+}(n,q)
\}_{j=0}^{N^+(n,q)}$ be that of $C(DC^+(n,q))$ $(n=2,4,6,\cdots$).
Then (a) For $j=0,\cdots,N^-(n,q)$,
\begin{align*}\label{}
  \begin{split}
 &C_{j}^{-}(n,q)=\sum \binom{q^{-1}A^-(n,q)(B^-(n,q)+1)}{\nu_0}\\
 &\times \prod_{tr(\beta^{-1})=0}\binom{q^{-1}A^-(n,q)(B^-(n,q)+q+1)}{\nu_{\beta }}\\
  &\times \prod_{tr(\beta^{-1})=1} \binom {q^{-1}A^-(n,q)(B^-(n,q)-q+1)}{\nu_{\beta}},
  \end{split}
\end{align*}
where the sum is over all the sets of nonnegative integers $\{\nu_
{\beta}\}_{\beta \in \fd_q}$ satisfying $\sum_{\beta \in \fd_q}\nu_
\beta=j$ and $\sum_{\beta \in \fd_q}\nu_{\beta}\beta=0$.

(b) For $j=0,\cdots,N^{+}(n,q)$,
\begin{align*}
  \begin{split}
 &C_{j}^{+}(n,q)\\
 &= \sum \binom{q^{-1}A^+(n,q)(B^+(n,q)+q^3-q^2-1)}{\nu_0} \prod_{\substack{|\tau |< 2\sqrt{q}\\ \tau \equiv -1(4)}}\\
 &\times \prod_{K(\lambda  ;  \beta^{-1})= \tau} \binom{q^{-1}A^+(A,q)(B^+(n,q)+q \tau -q^2 -1)} {\nu_ {\beta}},
  \end{split}
\end{align*}
where the sum is over all the sets of nonnegative integers $\{\nu_
{\beta}\}_{\beta \in \fd_q}$ satisfying $\sum_{\beta \in \fd_{q}}
\nu_{\beta}=j$, and $\sum_{\beta \in \fd_q} \nu_{\beta} \beta=0$,
and we assume $ q \geq 4$.
\end{theorem}

From now on, we will assume that, for $C(DC^-(n,q))^{\bot}$, either
$n \geq 3$ odd and all $q$ or $n=1$ and $q \geq 8$, and that, for
$C(DC^+(n,q))^{\bot}$, $ n \geq 2$ even and $q \geq 4$. Under these
assumptions, each codeword in $C(DC^{\mp}(n,q))^{\bot}$ can be
written as $c^{\mp}(a)$, for a unique $a \in \fd_q$(cf. Theorem
\ref{M}, (\ref{a37}), (\ref{a38})).

Now, we apply the Pless power moment identity in (\ref{a41}) to $
C(DC^{\mp}(n,q))^{\bot}$ for those values of $n$ and $q$, in order
to get the results in Theorem \ref{A} (cf. (\ref{a5}),
(\ref{a8}),(\ref{a9})) about recursive formulas.

The left hand side of that identity in (\ref{a41}) is equal to
\begin{equation*}
\sum_{a \in \fd_{q}^{*}}w(c^{\mp}(a))^h,
\end{equation*}
with $w(c^{\mp}(a))$ given by (\ref{a42})-(\ref{a44}). We have

\begin{equation*}
\sum_{a \in \fd_{q}^{*}}w(c^{-}(a))^{h}=\frac{1}{2^h}A^-(n,q)^h \sum
_{a \in \fd_{q}^{*}}(B^-(n,q)-K(\lambda ; a))^h
\end{equation*}
\begin{equation}\label{a47}
=\frac{1}{2^h}A^-(n,q)^h \sum_{l=0}^{h}(-1)^l
\binom{h}{l}B^-(n,q)^{h-l} M K^l.
\end{equation}
Similarly, we have
\begin{align}\label{a48}
  \begin{split}
&\sum_{a \in \fd_q^*}w(c^+(a))^h=\frac{1}{2^h}A^+(n,q)^h\\
&\times \sum_{l=0}^{h}(-1)^l \binom{h}{l}(B^+(n,q)-q^2+q)^{h-l} M
K^{2l}
  \end{split}
\end{align}
\begin{align}\label{a49}
  \begin{split}
=\frac{1}{2^h}A^+(n,q)^h  \sum_{l=0}^{h}(-1)^l \binom{h}{l}(B^+(n,q)
-q^2)^{h-l} M K_{2}^{l}.
  \end{split}
\end{align}

Note here that, in view of (\ref{a24}), obtaining power moments of
2-dimensional Kloosterman sums is equivalent to getting even power
moments of Kloosterman sums. Also, one has to separate the term
corresponding to $l=h$ in (\ref{a47})-(\ref{a49}), and notes $
dim_{\fd_2}C (DC^{\mp}(n,q))^{\bot}=r$.

%%%%%%%%%%%%%%%%%%%%%%%%%%%%%%%%%%%%%%%%%%%%%%%%%%%%%%%%%%%%%%%%%%%%%%%

%\begin{biography}[{\includegraphics[width=1in,height=1.25in,clip,keepaspectratio]{mshell}}]{Michael Shell}
% where an .eps filename suffix will be assumed under latex, and a .pdf suffix
% will be assumed for pdflatex; or if you just want to reserve a space for
% a photo:

%\begin{biography}{Michael Shell}
%\end{biography}

%\begin{biographynophoto}{John Doe}
%\end{biographynophoto}

% insert where needed to balance the two columns on the last page
%\newpage

%\begin{biographynophoto}{Jane Doe}
%\end{biographynophoto}


%\vfill


\end{document}